\documentclass{amsart}
\usepackage[english]{babel}
\usepackage[latin1]{inputenc}
\usepackage{mathrsfs}
\usepackage{amsmath}
\usepackage{amsthm}
\usepackage{amssymb}
\usepackage{indentfirst}

\usepackage{tikz}
\usetikzlibrary{shapes,fit,intersections}
\usetikzlibrary{decorations.markings}
\usetikzlibrary{decorations.pathreplacing}
\usetikzlibrary{patterns}
\usetikzlibrary{matrix,arrows}
\usepgflibrary{decorations.pathmorphing}

\usepackage{tikz-cd}

\theoremstyle{plain} 
\newtheorem{prop}{Proposition}[section]
\newtheorem{thm}[prop]{Theorem}
\newtheorem{lem}[prop]{Lemma} 
\newtheorem{cor}[prop]{Corollary}

\theoremstyle{remark}

\newtheorem{oss}[prop]{Remark}
\newtheorem{ex}[prop]{Example}

\theoremstyle{definition}
\newtheorem{defn}[prop]{Definition}
\newcommand{\hk}{hyperk\"{a}hler }

\newcommand{\kahl}{K\"{a}hler }

\newcommand{\ktipo}{$K3^{[2]}$ type}
\newcommand{\kntipo}{$K3^{[n]}$ type}
\newcommand{\kntiposp}{$K3^{[n]}$ type }

\newcommand{\Sym}{\text{Sym}}
\newcommand{\Aut}{\text{Aut}}
\newcommand{\Pic}{\text{Pic}}
\newcommand{\R}{\mathbb R}
\newcommand{\Teich}{\operatorname{Teich}}
\newcommand{\Comp}{\operatorname{Comp}}
\newcommand{\Diff}{\operatorname{Diff}}

\newcommand{\Perspace}{\operatorname{{\mathbb P}\sf er}}

\newcommand{\CC}{\mathbb{C}}
\newcommand{\Kth}{\textup{K3}}
\newcommand{\inv}{\imath}
\newcommand{\GL}{\textup{GL}}
\newcommand{\Lie}{\textup{Lie}}
\newcommand{\SL}{\textup{SL}}
\newcommand{\ZZ}{\mathbb{Z}}

\begin{document}
\title{Symplectic involutions of $K3^{[n]}$ type and Kummer $n$ type manifolds}
\author{Ljudmila Kamenova, Giovanni Mongardi, Alexei Oblomkov}
\address{Department of Mathematics, Stony Brook University, Math Tower 3-115, 
Stony Brook, NY 11794-3651, USA}
\email{kamenova@math.stonybrook.edu}
\address{Dipartimento di Matematica, Universit\'{a} degli studi di Bologna,  
Piazza di porta san Donato 5, Bologna, Italia 40126  }
\email{giovanni.mongardi2@unibo.it}
\address{Department of Mathematics and Statistics, University of Massachusetts, Lederle Graduate Tower, Amherst, MA 01003, USA}
\email{oblomkov@math.umass.edu}
\begin{abstract}
In this paper we describe the fixed locus of a symplectic involution on 
a hyperk\"ahler manifold of type $K3^{[n]}$ or of Kummer $n$ type. We prove 
that the fixed locus consists of finitely many copies of deformations of 
Hilbert schemes of $K3$ surfaces of lower dimensions and isolated fixed points. 
\end{abstract}

\maketitle

\section{Introduction}

An involution $\iota$ on a hyperk\"ahler manifold is symplectic if 
it preserves the holomorphic symplectic form.
Consider a hyperk\"ahler manifold $X$ of \kntipo, i.e., deformation 
equivalent to a Hilbert scheme of $n$ points on a $K3$ surface or of Kummer $n$ type, i.e. deformation equivalent to the Albanese fibre of the Hilbert scheme of $n+1$ points on an Abelian surface. 
We are interested in describing the fixed loci of symplectic involutions 
on $X$.\\
Automorphisms of \hk manifolds were studied by several authors, starting from the foundational work of Beauville \cite{Beau83}. As a non-exhaustive list of works let us mention the study of fixed point free automorphisms by Oguiso and Schr\"oer \cite{OS11}, and by Boissi\`{e}re, Nieper-Wisskirchen and Sarti \cite{BNWS12}, a full classification of symplectic actions on cohomology of fourfolds of $K3^{[2]}$ type by H\"ohn and Mason \cite{HM15} and by Boissi\`ere, Camere and Sarti for nonsymplectic actions \cite{BCS14}. 

 In \cite{nik} Nikulin proved 
that the fixed locus of a symplectic involution on a $K3$ surface consists 
of $8$ isolated fixed points. The second named author proved in \cite{mon1} 
that the fixed locus of a symplectic involution on a hyperk\"ahler 
manifold deformation equivalent to a Hilbert square of a $K3$ 
surface consists of $28$ isolated points and one $K3$ surface. 
This strengthens Camere's conjecture in \cite{cc}. 
Here we generalize these results to any hyperk\"ahler manifold of \kntipo. 

\begin{thm}
Let $X$ be a hyperk\"ahler manifold of \kntipo, and let $\iota$ be a
symplectic involution on $X$. Then the fixed locus $F$ of $\iota$ consists of 
finitely many copies of deformations of Hilbert schemes of $K3$ surfaces $S^{[m]}$ 
(where $m \leq \frac{n}{2}$) and possibly isolated fixed points 
(only when $n \leq 24$). The fixed locus $F$ 
is decomposed into loci of even dimensions $F_{2m}$, where $\max (0, 
\frac{n}{2} - 12) \leq m \leq \frac{n}{2}$. Each fixed locus $F_{2m}$ of 
dimension $2m$ has $$\sum_{2m=n-k-2l}{8 \choose k} { k \choose l}$$ connected components, each one of 
which is a deformation of a copy of $S^{[m]}$. In particular, 
the fixed locus $Z$ of largest dimension is the following.

({\it i}) If $n$ is even, then $Z$ consists of one copy of a deformation of $S^{[\frac{n}{2}]}$; 

({\it ii}) If $n$ is odd, then $Z$ consists of $8$ copies of a deformation of
$S^{[\frac{n-1}{2}]}$. 
\end{thm}

A key ingredient in the proof of this theorem is that the moduli space of 
pairs consisting of a \hk manifold of \kntiposp (or of Kummer $n$ type) together with a symplectic 
involution is connected. Using the Global Torelli theorem first we prove 
that $(X, \{\iota,Id_X\})$ is birational to a ``standard pair'', and then we prove 
that birational pairs can be deformed one into the other while preserving 
the group action. A ``standard pair'' for \kntiposp manifolds is a deformation of $(S^{[n]},G)$, 
where $S$ is a $K3$ surface and $G$ is a symplectic involution on $S^{[n]}$ 
induced by a symplectic involution on $S$. A group action on $X$ is called numerically standard (cf. Definition \ref{num_stand} for a precise statement) if its action on cohomology ``looks like'' a standard action. We prove the following: 

\begin{thm} 
Let $X$ be a manifold of \kntiposp or of Kummer $n$ type. Let $G\subset \mathrm{Aut}_s(X)$ be a 
finite group of numerically standard automorphisms. Then $(X,G)$ is a 
standard pair.
\end{thm}  

As a corollary to this theorem we obtain that the fixed locus of the 
symplectic involution $\iota$ has 
the form of the fixed locus on $S^{[n]}$ of a symplectic involution coming 
from the $K3$ surface $S$. Thus, we can restrict to the case 
where $X=S^{[n]}$ and $\iota$ is induced by an involution of $S$. Therefore, $\iota$ also defines an involution on 
$\Sym^n S$, compatible with the map $S^{[n]}\rightarrow S^{(n)}$. Depending on the parity of $n$, we can see that the number of 
irreducible components of the fixed locus of largest dimension is either 
one or eight. Using basic combinatorics, one can count the number of 
irreducible components of the fixed locus in each possible dimension. \\

In an analogous way, we compute the fixed locus of a symplectic involution on a \hk manifold of Kummer $n$ type, by reducing to the case of an involution coming from the sign change on the abelian surface $A$:

\begin{thm} 
Let $X$ be a hyperk\"ahler manifold of Kummer $n$ type, and let $\iota$ be a
symplectic involution on $X$ acting non-trivially on \(H_2(X)\). Then the fixed locus $F$ of $\iota$ consists of 
finitely many copies of deformations of Hilbert schemes of $K3$ surfaces 
$S^{[m]}$ (where $m \leq \frac{n+1}{2}$) and possibly isolated fixed points 
(only when $n \leq 48$). The fixed locus $F$ 
is decomposed into loci of even dimensions $F_{2m}$, where $\max (0, 
\frac{n+1}{2} - 24) \leq m \leq \frac{n+1}{2}$. Each fixed locus $F_{2m}$ of 
dimension $2m$ has $N^n_m$ 
connected components, each one of which is a deformation of a copy of $S^{[m]}$. 

\end{thm}

\section{Preliminaries}

Let $X$ be a hyperk\"ahler manifold, i.e., a complex K\"ahler simply 
connected manifold such that $H^{2,0}(X) \cong \mathbb C$ is generated by a 
holomorphic symplectic $2$-form. If $X$ is deformation equivalent to 
the Hilbert scheme $S^{[n]}$ of $n$ points of a $K3$ surface $S$, we say 
that $X$ is of \kntipo. If $X$ is deformation equivalent to the generalized Kummer $2n$-fold $K_n(A)$ of an abelian surface $A$, we say that $X$ is of Kummer $n$ type.

\begin{defn}
Let $X$ be a complex manifold and let $G$ be a subgroup of $\Aut(X)$, 
the group of automorphisms of $X$. A deformation of the pair $(X,G)$ 
consists of the following data: 

{\it (i)}  A flat family $\mathcal{X}\rightarrow B$, where $B$ is connected 
and $\mathcal{X}$ is smooth, and a distinguished point $0\in B$ such that 
$\mathcal{X}_0\cong X$.

{\it (ii)} A faithful action of the group $G$ on $\mathcal{X}$ inducing 
fibrewise faithful actions of $G$. 

Two pairs $(X,G)$ and $(Y,H)$ are deformation equivalent if $(Y,H)$ 
is a fibre of a deformation of the pair $(X,G)$. 

\end{defn}

In this paper we are mostly interested in deformations of the pair 
$(X, {\mathbb Z}_2)$, where $X$ is of \kntiposp and ${\mathbb Z}_2$ is 
generated by a symplectic involution. 

\begin{defn} 
Let $S$ be a $K3$ surface and let $G\subset \mathrm{Aut}_s(S)$ be a subgroup 
of the symplectic automorphisms on $S$. Then $G$ induces a subgroup of the 
symplectic morphisms on $S^{[n]}$ which we still denote by $G$. We call the 
pair $(S^{[n]},G)$ a natural pair as in \cite[Definition 1]{boi}. 
The pair $(X,H)$ is standard if it is deformation equivalent to a natural pair, as in \cite[Definition 1.2]{mon1}. 
If $A$ is an abelian surface, the same definitions apply to the generalized Kummer $2n$-fold $K_n(A)$ and symplectic automorphisms preserving $0\in A$, however the reader should notice that the induced action of $G$ on $H^2(K_n(A))$ is not necessarily faithful (while it stays faithful on $K_n(A)$), as there is a group of automorphisms acting trivially on the second cohomology, see \cite{BNS_enriques} .
\end{defn}
\begin{oss}
Notice that we stick with the convention used in \cite[Definition 1.2]{mon1}, where a pair $(X,G)$ is called standard if it can be deformed to a natural pair, instead of using the definition from \cite[Definition 4.1]{BCS14}, where the authors call such a pair natural as well, but notice that the two definitions are actually equivalent. 
\end{oss}
\begin{defn}
Let $G$ be a finite group acting faithfully on a manifold $X$. 
Define the invariant locus $T_G(X)$ inside $H^2(X, \mathbb Z)$ to be the 
fixed locus of the induced action of $G$ on the cohomology. The co-invariant 
lattice $S_G(X)$ is the orthogonal complement $T_G(X)^\perp$. The fixed locus 
of $G$ on $X$ is denoted by $X^G$. 
\end{defn}

As automorphisms of K3 and abelian surfaces are better known, it is interesting to determine whether an automorphism group on a manifold of \kntiposp (or of Kummer $n$ type) is standard or not. For related works in this direction, see \cite{BCS14} and \cite{jou}. We give the following criterion:
\begin{defn}\label{num_stand}
Let $Y$ be a manifold of \kntiposp or of Kummer $n$ type. A pair $(Y,H)$ is called numerically standard if the representation of $H$ on $H^2(Y,\mathbb{Z})$ is isomorphic to that of a standard pair $(X,H)$, up to conjugation by the monodromy group. More precisely, there exists a $K3$ (or abelian) surface $S$ with an $H$ action such that
\begin{itemize}
\item $S_H(S)\cong S_H(Y)$,
\item $T_H(S)\oplus \mathbb{Z}\delta=T_H(S^{[n]})\cong T_H(Y)$, where $2\delta$ is the class of the exceptional divisor of $S^{[n]}\rightarrow S^{(n)}$ (and analogously for the Kummer $n$ case),
\item The two isomorphisms above extend to isomorphisms of the Mukai lattices $U^4\oplus E_8(-1)^2$ (or $U^4$ in the Kummer case) after taking the canonical choice of an embedding of $H^2$ into the Mukai lattice described by Markman \cite[Section 9]{mar_tor} for the \kntiposp case and by Wieneck \cite[Theorem 4.1]{wie} for the Kummer case.
\end{itemize}
All the above isomorphisms are $H$-equivariant.
\end{defn}
This definition is slightly stronger than the one given in \cite[Definition 2.4]{mon2} for manifolds of \kntiposp, but they coincide when $n-1$ is a prime power, which was the case of interest in that paper. Notice that it is relatively easy to check the first two conditions, while the third one is more involved but, under some hypothesis on the group action, it is implied by the first two, see Proposition \ref{mukai_not}. \\ 

Let $X$ be a compact complex manifold and let 
$\Diff^0(X)$ be the connected component of its diffeomorphism 
group containing the identity. Denote by $\Comp$ the space of complex 
structures on $X$, equipped with the structure of a Fr\'echet manifold. 

\begin{defn}
The Teichm\"uller space of $X$ is the quotient  $\Teich:=\Comp/\Diff^0(X)$. 
\end{defn}

The Teichm\"uller space is finite-dimensional for a Calabi-Yau 
manifold $X$ (see \cite{cat}). 
Let $\Diff^+(X)$ be the group of orientable diffeomorphisms of 
a complex manifold $X$. The {\it mapping class group} 
$\Gamma:=\Diff^+(X)/\Diff^0(X)$ acts on $\Teich$. \\

By Huybrechts's result \cite[Theorem 4.3]{huyb}, 
non-separated points in the moduli space of marked hyperk\"ahler 
manifolds correspond to birational hyperk\"ahler manifolds. 
Consider the equivalence relation $\sim$ on $\Teich$ identifying 
non-separated points. Let $\Teich_b = \Teich/{}_\sim$ be the 
{\it birational Teichm\"uller space}. \\

Let $X$ be a hyperk\"ahler manifold, and let $\Teich$ be its Teichm\"uller 
space. Consider the map 
$\mathcal{P} : \Teich \rightarrow \mathbb PH^2(X, \mathbb C)$, 
sending a complex structure $J$ to the line 
$H^{2,0}(X,J) \in \mathbb PH^2(X, \mathbb C)$. 
The image of $\mathcal{P}$ is the open subset of a quadric, defined by 
$\Perspace:=\big\{l\in\mathbb PH^2(X,\mathbb C)\ \big|\ q(l,l)=0,\
q(l,\bar l)>0\big\}.$

\begin{defn}
The map $\mathcal{P}: \Teich \rightarrow \Perspace$ is called the period map, 
and the set $\Perspace$ is called the period domain. 
\end{defn}

The period domain $\Perspace$ is identified with the quotient 
$\frac{SO(3, b_2-3)}{SO(2) \times SO(1, b_2 -3)}$, see \cite[Proposition 3.1]{huy_tor}. 

\begin{thm} (Verbitsky's Global Torelli, \cite{verb}) \label{torelli}
The period map $\mathcal{P}: \Teich_b^0 \rightarrow \Perspace$ is an 
isomorphism on each connected component of $\Teich_b$. 
\end{thm}

It is possible to compute the \kahl cone of a \hk manifold from numerical data on the second cohomology (see \cite{av} and \cite{mon_ka} for the general theory), the following will be needed for our main result: 

\begin{prop}\label{walls_for_two}
Let $X$ be a manifold of \kntipo, $n$ odd. Let $a\in Pic(X)$ be a class of negative square with $a^2 \geq -6-2n$ and of divisibility two. Then there are no \kahl classes orthogonal to $a$.
\end{prop}
\begin{proof}
There is a canonical choice of an embedding of $H^2(X,\mathbb{Z})$ into the Mukai lattice $U^4\oplus E_8(-1)^2$ which is described by Markman \cite[Section 9]{mar_tor}. Let $\mathbb{Z}v:=(H^2)^\perp$ in this embedding. The lattice $L:=\langle v,a\rangle$ is generated by $v$ and $\frac{v+a}{2}$, whose square is at least $-2$ by hypothesis and at most $v^2/4$, therefore by \cite[Thm 5.7]{bm} $a^\perp$ is a wall for the space of positive classes, hence there cannot be a \kahl class orthogonal to it. 
\end{proof}

In particular, it follows from the results in \cite{av} and \cite{mon_ka} 
that if there is a \kahl class orthogonal to the Picard lattice, the \kahl 
cone coincides with the positive cone. \\

Some lattice theory will be used in the following, the main reference here is \cite{nik2}, where all of the following can be found. For a lattice $L$ we define the discriminant group $A_L:=L^\vee/L$. Let $l(A_L)$ denote the length of this group. If the lattice $L$ is even, $A_L$ has a bilinear form with values in $\mathbb{Q}/\mathbb{Z}$ induced from the bilinear form on $L$. This associated quadratic form is called the discriminant form of $L$ and is denoted by $q_{A_L}$. If $(l_+,l_-)$ is the signature of $L$, the integer $l_+-l_-$ is called signature of $q_{A_L}$ and, modulo $8$, it is well defined.\\

The following concerns primitive embeddings of lattices, i.e., embeddings where the quotient is torsion free:

\begin{lem}\cite[Proposition 1.15.1]{nik2}\label{lem:nik_immerge}
Let $S$ and $N$ be even lattices of signature $(s_+,s_-)$ resp. $(n_+,n_-)$. Primitive embeddings of S into $N$ are determined by the sets $(H_S,H_N,\gamma,K,\gamma_K)$, where $H_S$ and $H_N$ are subgroups of $A_S$ and $A_N$, $K$ is an even lattice with signature  $(n_+-s_+,n_--s_-)$ and discriminant form $-\delta$ where $\delta\,\cong\,(q_{A_S}\oplus -q_{A_N})_{|\Gamma_\gamma^\perp/\Gamma_\gamma}$ 
and $\gamma_K\,:\,q_K\,\rightarrow\,(-\delta)$ is an isometry. 
Moreover, two such sets $(H_S,H_N,\gamma,K,\gamma_K)$ and $(H'_S,H'_N,\gamma',K',\gamma'_K)$ determine isomorphic sublattices if and only if
\begin{itemize}
\item $H_S=\lambda H'_S$, $\lambda\in O(q_S)$,
\item $\exists\,\epsilon\,\in\,O(q_{A_N})$ and $\psi\,\in\,Isom(K,K')$ such that $\gamma'=\epsilon\circ\gamma$ and $\overline{\epsilon}\circ\gamma_K=\gamma'_K\circ\overline{\psi}$, where $\overline{\epsilon}$ and $\overline{\psi}$ are the isometries induced among discriminant groups.
\end{itemize} 
\end{lem}

Here $\Gamma_{\gamma}$ is the graph of $\gamma$.
When a lattice $L$ is a $G$-representation, we will call the sublattice 
$T_G(L)$ fixed by $L$ the invariant lattice, and its orthogonal complement 
$S_G(L)$ - the coinvariant lattice.

\section{Fixed loci of symplectic involutions}

We recall some properties of the irreducible components of the fixed locus 
of a symplectic involution. 

\begin{prop} \cite[Proposition 3]{cc} \label{propcc}
Let $X$ be a \hk manifold and $\iota$ be a symplectic involution on $X$. 
Then the irreducible components of the fixed locus of $\iota$ are 
symplectic subvarieties of $X$. 
\end{prop}


The moduli space of pairs consisting of a \hk manifold of \kntiposp together 
with a symplectic involution is connected. This follows from the following 
result of the second named author \cite[Theorem 2.5]{mon2}. Here we include 
a more general version of the result where we remove the assumption that 
$n-1$ is a prime power.

\begin{thm} \label{connected}
Let $X$ be a manifold of \kntiposp or of Kummer $n$ type. Let $G\subset \mathrm{Aut}_s(X)$ be a 
finite group of numerically standard automorphisms. Then $(X,G)$ is a 
standard pair.
\end{thm}  

\begin{proof}
First, we want to prove that $(X,G)$ is birational to a standard pair by using the Global Torelli theorem, and then we want to prove that birational pairs can be deformed one into the other while preserving the group action.
For the first step, up to deforming $X$, we can suppose that $\Pic(X):=S_G(X)\oplus \mathbb{Z}\delta$, where $\delta\subset T_G(X)$ is as in Definition \ref{num_stand}. Let $S$ be the $K3$ (resp. Abelian) surface with a $G$-action such that $NS(S)=S_G(S)=S_G(X)$ and $T_G(S)=T(S)=T(X)$. The $K3$ (resp. Abelian) surface $S$ is uniquely determined (resp. up to duality) if, under the identification $T(S)=T(X)$, we have the equality of the symplectic forms $\sigma_S=\sigma_X$ in $T(X)\otimes\mathbb{C}$ (here we can take the normalized symplectic forms, so that $\int_X (\sigma_X\overline{\sigma}_X)^{\mathrm{dim}(X)}=\int_S \sigma_S\overline{\sigma}_S=1)$. We have therefore $H^2(S)=\delta^\perp$ under the above identification (resp. $K_n(S)$) and $X$ and $S^{[n]}$ are Hodge isometric as we have $H^2(S^{[n]})=\delta^\perp\oplus \mathbb{Z}\delta$. By the requirement of Definition \ref{num_stand}, this Hodge isometry extends to an isometry of the Mukai lattice, so with a suitable choice of markings $f,g$ the pairs $(X,f)$ and $(S^{[n]},g)$ (resp. $K_n(S)$) are in the same component of the Teichmuller space of $X$, thus by Theorem \ref{torelli} they are birational and the birational map commutes with the $G$-action by our construction of the Hodge isometry.

Now we continue with the second step. Let $U$ be a small open subset of $H^1(\mathcal{T}_X)$, that is a family of local deformations of $X$, with total family $\mathcal{X}$. By \cite[Theorem 4.3 and proposition 4.5]{huyb} there exists $V$ and a local deformation of $Y$ over $V$, with total family $\mathcal{Y}$ such that $U\cong V$ and a fibrewise birational map $\varphi:\,\mathcal{X}\dashrightarrow \mathcal{Y}$. Let us restrict to the local deformations of the pairs $(X,G)$ and $(Y,G)$, which are the families over $U^G$ (which coincides with $V^G$, as the isomorphism $U\cong V$ is a restriction of the $G$-equivariant isomorphism $H^1(\mathcal{T}_X)\cong H^1(\mathcal{T}_Y)$ of first order deformations). Let $t\in U^G$ be a point such that $\Pic(\mathcal{X}_t)=S_G(\mathcal{X}_t)$, which is true for very general points as the locus where the Picard number increases are union of closed subsets. There is a \kahl class orthogonal to $S_G(\mathcal{X}_t)$ and this class lies in the orthogonal complement to $\Pic(\mathcal{X}_t)$. Thus, the \kahl cone is the full positive cone and all manifolds birational to $\mathcal{X}_t$ are isomorphic to it, so $\mathcal{X}_t\cong \mathcal{Y}_t$. This isomorphism acts on the second cohomology by extending the identifications $H^1(\mathcal{T}_X)\cong H^1(\mathcal{T}_Y)$, which contain the non trivial part of the $G$ action, therefore this isomorphism is also $G$-equivariant. This finally implies that the pairs $(X,G)$ and $(Y,G)$ are both deformation equivalent to $(\mathcal{X}_t,G)$, and our claim holds. 
\end{proof}
For further details on deformations of pairs $(X,G)$ and their moduli spaces, the interested reader can consult the recent \cite[Section 3.1 and 3.2]{BC}.

\begin{cor}\label{one_invol}
Let $X$ be a manifold of \kntiposp and let $\iota$ be a symplectic involution. 
Then the fixed locus of $\iota$ is a deformation of the fixed locus on 
$S^{[n]}$ of a symplectic involution coming from the $K3$ surface $S$.  
\end{cor}
\begin{proof}
By \cite[Cor. 37]{mon3}, the coinvariant lattice of a symplectic involution is always isometric to $E_8(-2)$. Thus, by Proposition \ref{count_embed}, there is an embedding of $E_8(-2)$ inside $H^2(X)$ which is numerically standard and all others have an element $v$ inside $E_8(-2)$ which has divisibility $2$ and is of square at least $-6-2n$. The latter case cannot be induced by an involution on a manifold of \kntipo, because otherwise an invariant \kahl class would be orthogonal to $v$ and this is in contradiction with Proposition \ref{walls_for_two}.  
Thus, all pairs $(X,\iota)$ are numerically standard, therefore Thm. \ref{connected} applies and we obtain our claim.
\end{proof}

\begin{cor}\label{one_invol_kum}
Let $X$ be a manifold of Kummer $n$ type and let $\iota$ be a symplectic involution. 
Then the fixed locus of $\iota$ is a deformation of the fixed locus on 
$K_n(A)$ of the $-1$ involution coming from the Abelian surface $A$.  

\end{cor}
\begin{proof}
All possible symplectic involutions on the second cohomology of a manifold of Kummer $n$ type $X$ have been classified in \cite[Section 5 and 6]{mtw}. Notice that there is only one such involution by \cite[Prop. 6.1]{mtw}, which however acts as an order four automorphism on $X$, therefore a symplectic involution on $X$ must have a trivial action on $H^2(X)$. By Theorem \ref{connected}, the pair $(X,\iota)$ is deformation equivalent to the pair $(K_n(A),-1)$ and our claim holds.
\end{proof}

Thus we see that if \(\iota\) acts trivially on \(H_2(X)\) then \(n\) is odd and the involution is induced by the shift \(\tau\) by an order element of \(A\).
The quotient of \(A\) by the action of \(\tau\) is an abelian variety \(\tilde{A}\). The fixed locus of \(\iota\) consists of \(8\) copies \(K_{(n-1)/2}(\tilde{A})\).
Indeed, in the case \(n=2\) the fixed locus consist unordered  pairs of points \((z,z+\tau)\), \(z\in A\) that satisfy equation
\(2z+\tau=0\). There are exactly \(8\) solutions for the last equation in \(\tilde{A}\).

For the general \(n\) the fixed locus is the closure inside \(A^{[n]}\)
of set of unordered
distinct \(n\)-tuples \((z_1,z_1+\tau,\dots,z_{(n-1)/2},z_{(n-1)/2}+\tau)\), \(z_i\in\) that satisfy equation \(2(z_1+\dots+z_{(n-1)/2})=\tau(1-n)/2\).
The last equation translates to the equation on \(\tilde{A}\) of the form \(\tilde{z}_1+\dots+\tilde{z}_{(n-1)/2}=\xi_k\), \(k=1,\dots,8\) where \(\xi_k\) are the images of
the order two points under the quotient map \(A\to \tilde{A}\).

\begin{ex}
Consider $\Sym^2(S)$, where $S$ is a $K3$ surface with a symplectic involution 
$i$. The involution $i$ induces an involution $\iota$ on $\Sym^2(S)$. 
The fixed locus of $\iota$ has a component of the form $S/i$, because 
locally it consists of unordered pairs of points $\{p, i(p)\}$, where 
$p \in S$. 
The rest of the fixed points are isolated unordered pairs of the form 
$\{p, q\}$, where $p$ and $q$ are fixed points of $i$ on $S$, with 
possible repetitions. In \cite{nik} Nikulin proved that $i$ has $8$ fixed 
points on $S$. In \cite{mon1} the second named author proved that 
the fixed locus of $\iota$ on $S^{[2]}$ consists of $28 = \binom{8}{2}$ 
isolated points and one copy of $S$. The eight fixed points of type 
$\{p, p\}$ are contained in the minimal resolution of $S/i$ which is $S$. 
\end{ex}

\begin{ex}
Consider $\Sym^3(S)$, where $S$ is a $K3$ surface with a symplectic involution 
$i$. The involution $i$ induces an involution $\iota$ on $\Sym^3(S)$. 
The fixed locus of largest dimension of $\iota$ locally looks like 
$\{p, i(p), q \}$, where $p \in S$ and $q$ is a fixed point of $i$ on $S$. 
There are $8$ connected components of fixed loci of the form $S/i$, 
because there are $8$ possibilities for $q$ by Nikulin's result \cite{nik}. 
The rest of the fixed points are isolated of the form $\{p, q, r \}$, 
where $p$, $q$ and $r$ are fixed points of $i$ on $S$. In total, there are 
$56$ points on $S^{[3]}$ corresponding to triples consisting of three 
different points $\{p, q, r \}$, and all fixed points of the form $\{p,p,q\}$ 
with $p\neq q$ are contained in the resolution of $S/\iota$. There are eight 
more isolated fixed points given by schemes fully supported on one point whose 
reduced scheme structure contains all possible tangent directions. 
\end{ex} 

The two examples above illustrate the difference between the even and the 
odd cases of $n$ when considering symplectic involutions on $S^{[n]}$ and 
are indicative of the approach towards the following theorem. 

\begin{thm}\label{counting}
Let $X$ be a hyperk\"ahler manifold of \kntipo, and let $\iota$ be a
symplectic involution on $X$. Then, up to deformation, the fixed locus $F$ of $\iota$ consists of 
finitely many copies of Hilbert schemes of $K3$ surfaces $S^{[m]}$ 
(where $m \leq \frac{n}{2}$) and possibly isolated fixed points 
(only when $n \leq 24$). The fixed locus $F$ 
is decomposed into loci of even dimensions $F_{2m}$, where $\max (0, 
\frac{n}{2} - 12) \leq m \leq \frac{n}{2}$. Each fixed locus $F_{2m}$ of 
dimension $2m$ has
$$\sum_{2m=n-k-2l}{8 \choose k} { k \choose l}$$ connected components, each one of 
which is a deformation of a copy of $Y^{[m]}$, where $Y$ is the $K3$ resolution of $S/\iota$. In particular, 
the fixed locus $Z$ of largest dimension is the following.

({\it i}) If $n$ is even, then $Z$ consists of one copy of $Y^{[\frac{n}{2}]}$; 

({\it ii}) If $n$ is odd, then $Z$ consists of $8$ copies of 
$Y^{[\frac{n-1}{2}]}$. 
\end{thm}

\begin{proof}
From Corollary \ref{one_invol}, we can restrict to the case when $X=S^{[n]}$ and 
$\iota$ comes from a symplectic involution on $S$. As the fixed locus of an automorphism deforms smoothly in an equivariant family, our claim will follow if we prove it in this setting. Let us denote by $Y$ the minimal resolution of $S/\iota$. We will also prove that all components of the fixed locus are Hilbert schemes over $Y$. 
Since $\iota$ on $X$ comes from an automorphism of $S$, it also factors through the Hilbert-Chow morphism and it is defined also on $\Sym^n S$. 

By Proposition \ref{propcc}, the irreducible components of the fixed locus of 
$\iota$ are symplectic subvarieties of $X$, and in this case each one of them 
has even dimension $2m$. Let us label the fixed locus of dimension $2m$ by 
$F_{2m}$. By Nikulin's theorem in \cite{nik}, the symplectic involution on 
the $K3$ has $8$ fixed points $f_1, \cdots, f_8$. A fixed point of $\iota$ on $\Sym^n(S)$ is therefore supported on the fixed points and on pairs of points $(p,\iota(p))$. Therefore, on $\Sym^n(S)$ each component 
$F_{2m}$ is isomorphic to $\Sym^m S/\iota \times \Sym^{n-2m}(f_1 \cup \cdots \cup f_8)$. 
Let $l_1, \cdots, l_8$ be the degrees of $f_1, \cdots, f_8$ in $F_{2m}$, that is their coefficients in $\Sym^{n-2m}(f_1 \cup \cdots \cup f_8)$.

Let us denote by $U_i$ some small analytic neighborhoods around $f_i$. Since $S$ is connected and by the structure of $F_{2m}$,
we can connect any involution fixed point on $\Sym^n(S)$ to a point inside $(U_1)^{[n_1]}\times\dots\times (U_8)^{[n_8]}$
for some $n_i\ge 0$, by moving the points inside $\Sym^m S/\iota$. Respectively, we can now use the computations of Appendix B 
in this analytic neighbourhoods, so that the fixed locus inside them are the products
$U_{\vec{n},\vec{s}}=\left((U_1)^{[n_1]}\right)_{s_1}^\inv\times\dots\times \left((U_8)^{[n_8]}\right)_{s_8}^\inv$
where $s_i$ is $\pm$ if $n_i$ is odd and $s_i=\emptyset$ if $n_i$ is even (see Lemma \ref{lem:quiv} and  the last formula in  Appendix B).
Let $k(\vec{n},\vec{s})$ be a number of odd $n_i$  and $l(\vec{n},\vec{s})$ is the number of $i$ such that $s_i=-$. Then the dimension of
$U_{\vec{n},\vec{s}}$ is $n-k(\vec{n},\vec{s})-2l(\vec{n},\vec{s})$.

Analogously, we can move a pair of points $(z,\iota(z))$ from one neighborhood $U_i$ to another $U_j$, thus we connect
analytic sets $U_{\vec{n},\vec{s}}$ and $ U_{\vec{n'},\vec{s'}}$ as long as $k(\vec{n},\vec{s})=k(\vec{n'},\vec{s'})$ and
$l(\vec{n},\vec{s})=l(\vec{n'},\vec{s'})$. On the other hand, it's also clear that the if we can connect
analytic sets $U_{\vec{n},\vec{s}}$ and $ U_{\vec{n'},\vec{s'}}$ then $k(\vec{n},\vec{s})=k(\vec{n'},\vec{s'})$, because we
can only move points between the neighborhoods in pairs. Finally, if the invariant $l(\cdot,\cdot)$ changes
along a path then the dimension of the connected component would change too. Thus our formula for the number of connected components of $F_{2m}$ reduces to count pairs $k(\vec{n},\vec{s})$ and $l(\vec{n},\vec{s})$ such that $2m=n-k(\vec{n},\vec{s})-2l(\vec{n},\vec{s})$. For a fixed numer $k=k(\vec{n},\vec{s})$, there are ${k\choose l}$ possibilities for $(\vec{n},\vec{s})$ that give $l=l(\vec{n},\vec{s})$ and clearly for every $k$ there are ${8\choose k}$ possibilities for $(\vec{n},\vec{s})$ that give $k=k(\vec{n},\vec{s})$. Therefore we proved our formula
for the number of connected components.


Now we shall describe explicitely the fixed locus $Z = F_{2[\frac{n}{2}]}$ 
of largest dimension. Let $m = [\frac{n}{2}]$ be the largest integer not 
greater than $\frac{n}{2}$, i.e., 
$m=\frac{n}{2}$ if $n$ is even and $m=\frac{n-1}{2}$ if $n$ is odd. 
Then the fixed locus of largest dimension is of the form 
$\{ x_1, \iota(x_1), x_2, \iota(x_2), \cdots, x_m, \iota(x_m)  \}$ if 
$n$ is even, i.e., one copy of $\Sym^m S/\iota$, where $x_i \in S$ for 
$1 \leq i \leq m$. In the case when $n$ is odd, 
the fixed locus of largest dimension is of the form 
$\{ x_1, \iota(x_1), x_2, \iota(x_2), \cdots, x_m, \iota(x_m), x_{m+1} \}$, 
where $x_i \in S$ for $1 \leq i \leq m$ and $x_{m+1}$ is a fixed point, i.e., 
$Z$ contains eight copies of 
$\Sym^m S/\iota$ since there are $8$ choices for 
$x_{m+1}$. In both cases, the dimension of $\Sym^m S/\iota$ is $2m$. 
\end{proof}


\begin{oss}
  The involution fixed locus does not have a zero dimensional component if 
$n>24$. On the other hand, 
  $\left(\Kth^{[24]}\right)^{\inv}$ has a zero dimensional component which is 
the product of eight squares of maximal ideals of the involution fixed locus. 
\end{oss}

Finally, let us conclude with the analogous result in the Kummer case. To state the combinatorial part of the result
we need a few extra notations. Given a subset $I\subset \ZZ_2^4$, we denote by $|I|$ the size of the set and $||I||$ the element
of $\ZZ^4_2$ which is the product of all the elements in the set. Thus, 
we define:
$$N^n_m=\sum_{I,||I||=1}{(n-|I|)/2-m \choose |I|}.$$

\begin{thm}\label{counting_kum}
Let $X$ be a hyperk\"ahler manifold of Kummer $n$ type, and let $\iota$ be a
symplectic involution on $X$ that acts nontrivially on \(H_2(X)\).
Then, up to deformation, the fixed locus $F$ of $\iota$ consists of 
finitely many copies of Hilbert schemes of $K3$ surfaces $S^{[m]}$ 
(where $m \leq \frac{n+1}{2}$) and possibly isolated fixed points 
(only when $n \leq 48$). The fixed locus $F$ 
is stratified into loci of even dimensions $F_{2m}$, where $\max (0, 
\frac{n+1}{2} - 24) \leq m \leq \frac{n+1}{2}$. Each fixed locus $F_{2m}$ of 
dimension $2m$ has $N_m^{n+1}$ connected components, each one of 
which is a deformation of a copy of $S^{[m]}$. In particular, 
the fixed locus $Z$ of largest dimension is one copy of $S^{[\frac{n+1}{2}]}$.
\end{thm}

\begin{proof}
From Corollary \ref{one_invol_kum}, we can restrict to the case when $X=K_n(A)$ and 
$\iota$ comes from a the $-1$ involution on $A$. 
Since the involution $\iota$ acts as the identity on the classes 
in $H^2(X, \R)$, it preserves the exceptional divisor of $K_n(A)$. 
Therefore, it descends to an action on $\Sym^{n+1} A$ preserving the fibre over $0$ of the Albanese map. Let us first look at this action on $\Sym^{n+1} A$, and then let us look at the Albanese fibre of the fixed locus. 

By Proposition \ref{propcc}, the irreducible components of the fixed locus of 
$\iota$ are symplectic subvarieties of $X$, and in this case each one of them 
has even dimension $2m+2$. Let us label the fixed locus of dimension $2m+2$ by 
$F_{2m}$ (whose structure in $\Sym^{n+1}$ is the same of $F_{2m+2}$ in the previous Theorem). The symplectic involution on 
the Abelian surface $A$ has $16$ fixed points $f_1, \cdots, f_{16}$, and there is a natural identification
of $\{f_1,\dots,f_{16}\}$ with $\ZZ_2^4$. The same argument as in the previous theorem implies that there is
a natural correspondence between the $2m$-dimensional connected components of the involution fixed locus and the
set of pairs: $I_-\subset I_{odd}\subset \ZZ^4_2$ such that $2m=n-|S_{odd}|-2|S_-|$.

Finally, let us notice that the Albanese map is constant on the connected 
components and the value of the map on the connected component is $1$ 
iff $||I_{odd}||=1$. 

\end{proof}

\section*{Acknowledgements}
The first named author is partially supported by a grant from the 
Simons Foundation/SFARI (522730, LK). She thanks Misha Verbitsky for their 
interesting conversations about this result. The second named author was supported by ``National Group for Algebraic and Geometric Structures, and their Application'' (GNSAGA - INdAM). The third author was supported by  the NSF CAREER grant DMS-1352398, NSF FRG grant
DMS-1760373 and Simons Fellowship. The author thanks Professor Hiraku Nakajima for conversations and help with the references. We also thank the anonymous referee for their helpful list of suggestions and improvements to an earlier version of this paper. 

\section*{Appendix A: lattice computations}
In this appendix we include all the computations needed in the proof of Corollary \ref{one_invol}.
Let $E_8$ be the unique unimodular even positive definite lattice of rank 8 and let $S:=E_8(-2)$ be the same lattice with the quadratic form multiplied by $-2$. The discriminant group of $S$ is $\mathbb{Z}_2^8$ and its elements are classes $[v/2]$, where $v\in S$. The discriminant form on these elements takes the values $0$ or $1$ modulo $2\mathbb{Z}$ (i.e., it is $0$ if $v^2$ is divisible by $8$ and $1$ otherwise). We have the following:

\begin{lem}\label{small_square}
For every element $\alpha$ in $A_S$ there is an element $v\in S$ such that $[v/2]=\alpha$ and $v^2\geq -16$.
\end{lem}

\begin{proof}
As $E_8$ is generated by its roots, the discriminant group of $S$ is generated by classes of half-roots, which have square $-4$. Thus, all elements of $A_S$ can be represented by half the sum of at most eight distinct roots. The sum of two roots is either a root (if they are not orthogonal) or an element of square $-8$. As there can be at most a set of four orthogonal roots, the claim follows.
\end{proof}

Let $L_n:=U^3\oplus E_8(-1)^2\oplus (-2n+2)$ be the lattice corresponding to the second cohomology of a \kntiposp manifold. We then have the following:

\begin{prop}\label{count_embed}
Let $L_n$ and $S$ be as above. Then, up to isometry, there is only one primitive embedding of $S$ into $L_n$ such that $S$ contains no element of divisibility two and square bigger than $-6-2n$.
\end{prop}

\begin{proof}
The discriminant group of $L_n$ has one generator whose discriminant form is $-\frac{1}{2(n-1)}$. There is only one element in this group whose order is precisely two, and it has discriminant form $-\frac{n-1}{2}$ modulo $2\mathbb{Z}$. As per Lemma \ref{lem:nik_immerge}, primitive embedding of $S$ into $L_n$ are determined by a quintuple $(H_S,H_{L_n},\gamma,K,\gamma_K)$, where the first two are subgroups of $A_S$ and $A_{L_n}$, respectively, and $\gamma$ is an anti-isometry between the two. Thus, when $n$ is even, $H_S$, $H_{L_n}$ and $\gamma$ are trivial as all elements of $A_S$ have integer square. When $n$ is odd, either we are in the same case as before or we have nontrivial $H_S$, $H_{L_n}$ and $\gamma$. In the latter case, $H_{L_n}$ is unambiguously determined and by Lemma \ref{small_square} the non trivial element of $H_S$ is represented by half an element $v$ of square at least $-16$ (more specifically, at least $-16$ if $n-1$ is divisible by four and at least $-12$ otherwise). Thus, $v$ is an element of square at least $-6-2n$ and of divisibility $2$ in $L_n$, as $[v/2]$ is non trivial in $A_{L_n}$.
\end{proof}

We conclude this appendix with a criterion to avoid checking the last condition in Definition \ref{num_stand}:

\begin{prop}\label{mukai_not}
Let $(X,G)$ be a pair such that there exists a $K3$ (resp. Abelian) surface $S$ and $G\subset Aut_s(S)$ such that $H^2(S^{[n]})$ (resp. $H^2(K_n(A))$) and $H^2(X)$ are isomorphic $G$ representations. Moreover, suppose that $U\subset T_G(S)$. Then $(X,G)$ is numerically standard.
\end{prop}

\begin{proof}
By the hypothesis we have a Hodge isometry $\varphi\,:\,H^2(X) \rightarrow H^2(S^{[n]})$ (resp. $H^2(K_n(A))$) which might not extend to an isometry of the Mukai lattice $\Lambda:=U^4\oplus E_8(-1)^2$ (resp. $\Lambda:=U^4$), that is, it is not compatible with the two embeddings $\psi_1\,:\,H^2(X,\mathbb{Z})\rightarrow \Lambda$ and $\psi_2\,:\,H^2(S^{[n]},\mathbb{Z})\rightarrow \Lambda$. Let $\delta$ be half the class of the exceptional divisor in $S^{[n]}$ (resp. $K_n(A)$) and let $\delta_x:=\varphi^{[-1]}(\delta)$. Let $v$ be a generator of $\psi_2(H^2)^\perp$ and $v_x$ be the same for $\psi_1(H^2)^\perp$. As discussed in in \cite[Section 9]{mar_tor}, the fact that $\varphi$ does not extend to $\Lambda$ means that it does not respect the two gluing data associated to the pairs $(v,\delta)$ and $(v_x,\delta_x)$. The gluing data corresponds to a choice of an anti-isometry between the two discriminant groups $A_{H^2}$ and $A_{\langle v\rangle}$. However, we have $U\subset T_G(S)$ and let $L:=U\oplus \mathbb{Z}\delta\subset T_G(S^{[n]})$ (resp. $T_G(K_n(A))$), thus, by \cite[Thm 1.14.2]{nik2} applied to $L$, we can compose $\varphi$ with an isometry $\gamma$ of $H^2(S^{[n]})$ (resp. $H^2(K_n(A))$) which is trivial on $L^\perp$ and arbitrarily non trivial on $A_L\cong A_{H^2(S^{[n]})}$, thus $\delta\circ \varphi$ extends to an isometry of the Mukai lattices. 

\end{proof}

In particular, this proposition applies to symplectic involutions on manifolds of \kntiposp.

\section*{Appendix B: fixed loci, local case}
\label{sec:irred}
In the proof of theorem~\ref{counting} we use the Hilbert-Chow map \(\Kth^{[n]}\to \Sym^n(\Kth)\). The map commutes with the involution \(\inv\).
For our proof we need to understand the structure of the connected components of the fibers of the Hilbert-Chow map over the quotient \(U/\inv\) a small analytic neighborhood of
a point \[\prod_{i=1}^k\Sym^{2n_i}(p_i,p_i,\dots,\inv(p_i))\times\prod_{j=1}^8 \Sym^{m_j}(f_j\dots,f_j).\]
 with \(\inv(p_i)\ne p_i\), \(p_i\ne p_m\), \(i\ne m\) ,\(i,m=1,\dots,k\) and \(f_j\) are fixed point of \(\inv\).

There are \(\inv\)-equivariant analytic neighborhoods \(p_i\cup \inv(p_i)\subset U_{p_i}\subset\Kth \), \(f_j\in V_{f_j}\subset \Kth\) such that \(\inv\)-fixed locus
 \(U^\inv\) is naturally a product \(\Sym^{n_k}(U_{p_i}/\inv)\) and \(\Sym^{m_j}(V_{f_j}^\inv)\). The fibers of the Hilbert-Chow map over the first type
are the products of the punctual Hilbert schemes. In particular,  these
fibers are connected and have known dimension.  The fibers over the second type of neighborhood require need to be studied. Since \(V_{f_j}\) is analytically isomorphic to an analytic subset of \(\CC^2\) we can identify the fibers with with fibers
of the corresponding Hilbert-Chow map for \(\CC^2\).

The surface \(\CC^2\) has a natural symplectic form \(\omega=dx\wedge dy\) and the involution
\[\inv: x\mapsto -x,\quad y\mapsto -y\]
preserves this form. The quotient of \(\CC^2\) by the involution is a singular surface that admits
a symplectic resolution:
\[\widehat{\mathrm{A}}_1\to\mathrm{A}_1=\CC^2/\inv.\]

It is elementary to see that the \(\inv\)-fixed locus on \((\CC^2)^{[2]}\) is isomorphic to \(\widehat{\mathrm{A}}_1\).
Below we show a generalization of this statement:
\begin{lem}\label{lem:quiv}
  For any \(n\) we have 
  \[\left((\CC^2)^{[n]}\right)^\inv=(\widehat{A}_1)^{ k},\]
  if $n=2k$ is even and for odd $n=2k+1$:
  $$\left((\CC^2)^{[n]}\right)^{\inv}=\mathfrak{M}(k\delta+e_1,e_1)\bigcup \mathfrak{M}(k\delta+e_2,e_1)$$
  where the $\mathfrak{M}(v,w)$ is the quiver variety for the quiver of the affine Dynkin diagram of type
  $\tilde{A}_1$ and $\delta=e_1+e_2$ is the imaginary root of the corresponding root system.
  The quiver varieties are connected and of dimension:
 $$  \dim\left(\mathfrak{M}(k\delta+e_1,e_1)\right)=2k,\quad\dim\left( \mathfrak{M}(k\delta+e_2,e_1)\right)=2k-2.$$
\end{lem}

To explain the statement we need to remind our reader some basics of the theory of the quiver varieties \cite{Nakajima98}.
A quiver is a directed graph $Q$ with a set of vertices $I$. Given $\alpha\in \mathbb{N}^I$, the set of representations of the
quiver is:
$$\mathrm{Rep}(Q,\alpha)=\oplus_{a\in Q} \mathrm{Mat}(\alpha_{h(a)}\times \alpha_{t(a)}),$$
where $h(a)$ and $t(a)$ are the head and the tail of the corresponding arrow. The group $$G(\alpha)=\left(\prod_{i\in I}\GL(\alpha_i)\right)/\CC^*$$
acts on the vector space of representations.

The cotangent bundle of the space of representations is the space of the representations of the  {\it double} of the quiver $\bar{Q}$:
$$\mathrm{T}^*\mathrm{Rep}=\mathrm{Rep}(\bar{Q},\alpha),$$
where the double is the quiver obtained from $Q$ by adjoining a reverse arrow $a^*$ for every arrow $a\in Q$.

The moment map $\mu:\mathrm{Rep}(\bar{Q},\alpha)\to \Lie(G(\alpha))$ is given by:
$$\mu(x)_i=\sum_{h(a)=i}x_ax_{a^*}-\sum_{t(a)=i}x_{a^*}x_a.$$

Let $Q_0$ be a quiver and $u,v\in \mathbb{N}^{I_0}$. We define $Q$ to be the quiver with the set of vertices $I=I_0\cup \infty$ and
the set of arrows is the union of the set of arrows of $Q_0$ and the arrows $v_i$ from  $\infty$ to $i\in Q_0$. Respectively,
we define:
$$\mathbf{M}(v,w)=\mathrm{Rep}(\bar{Q},\alpha),\quad G_v=G(\alpha),$$
where $\alpha$ is the vector with coordinates: $\alpha_i=u_i$, $i\in I_0$ and $\alpha_\infty=1$.
Nakajima \cite{Nakajima98} defines the quiver variety as the GIT quotient of the subvariety of  $\mathbf{M}(v,w)$:
$$\mathfrak{M}(v,w)=\mu^{-1}(0)//(G_v,\chi),$$
where $\chi$ is the character of the group defined by $\chi(g)=\prod_{k\in I}\det(g_k^{-1})$. We indicate the dependence of
the quiver variety on the underlying quiver by the subindex: $\mathfrak{M}_{Q_0}(v,w)$, $\mathbf{M}_{Q_0}(v,w)$.

In our study we are most interested in the quiver varieties associated to the following two quivers:
\[\mathbf{Q}_0=\begin{tikzcd}
  \arrow[l,loop left]\bullet 
\end{tikzcd}, \quad
\mathbf{Q}'_0=
 \begin{tikzcd}
\bullet\arrow[r,bend right]\arrow[r,bend left]&\bullet
\end{tikzcd}
\]

these are the quivers of affine Dynkin diagrams of types $\tilde{A}_0$ and $\tilde{A}_1$. Let the dimension vectors be
$(v,w)=((n),(1))$, $(v,w)=((n_1,n_2),(0,1))$, then
the corresponding enhanced quivers are:

\[\mathbf{Q}=\begin{tikzcd}
  \arrow[l,loop left]1 &\arrow[l]\infty
\end{tikzcd}, \quad
\mathbf{Q}'=
 \begin{tikzcd}
1\arrow[r,bend right]\arrow[r,bend left]&2&\arrow[l]\infty
\end{tikzcd},
\]
in the pictures of the quivers we introduced the labels of the vertices.

The  starting point of our proof is the quiver description of the space $\left(\CC^2\right)^{[n]}$ \cite{Nakajima99}:
$$\mathfrak{M}_{\mathbf{Q}_0}((n),(1))=\left(\CC^2\right)^{[n]}.$$

\begin{proof}[Proof of lemma \ref{lem:quiv}]

The involution \(\inv\) on \(\CC^2\) induces an action on the \(\mathbf{M}_{\mathbf{Q}_0}((n),(1))\),  this involution acts trivially on
\(\CC^1\) which corresponds to the vertex $\infty$, and
the space \(\CC^n\) corresponding to the vertex $1$ decomposes into the 
anti-invariant and invariant parts
\(\CC^n=\CC^{n_1}\oplus \CC^{n_2}\).

The vector space $\CC^n$ in the quiver description of the Hilbert scheme corresponds to the quotient space
$\CC[x,y]/I$ in the interpretation of the natural Hilbert scheme. Moreover, the image of the map corresponding to the arrow from $\infty$ to $1$
is the span of  $1\in \CC[x,y]/I$, hence the image is invariant under the involution $\inv$. Thus, the involution invariant part of the
quiver variety union of the quiver varieties constructed from the quiver representations is of the form:

\[
 \begin{tikzcd}
  \CC^{n_1}\arrow[r,bend right]\arrow[r,bend left]&\CC^{n_2}\arrow[r]&\CC^1
\end{tikzcd}.
\]

More formally, we conclude that we have an inclusion:
$$\left(\left(\CC^2\right)^{[n]}\right)^{\inv}\subset \bigcup_{n_1+n_2=n}\mathfrak{M}_{\mathbf{Q}'_0}((n_1,n_2),(0,1)).$$

Next we recall the result of \cite{CrawleyBoevey01} that concerns with the classification of connected
non-empty quiver varieties. The result from the last part of the introduction in \cite{CrawleyBoevey01} states
that the quiver variety $\mathfrak{M}_{Q_0}(v,w)$ is non-empty if and only if $v$ is a positive root of
the Kac-Moody Lie algebra corresponding to the quiver $Q_0$ and if it is non-empty, then it is connected.

The Kac-Moody Lie algebra corresponding to the quiver $\mathbf{Q}'_0$ is the Lie algebra of the loop group
of $\SL_2$ and the roots of this Lie algebra are:
$$n\delta,\quad e_1+n\delta,\quad e_2+n\delta,$$
here $\delta=e_1+e_2$.

Thus, we can refine our previous inclusion:
$$\left(\left(\CC^2\right)^{[n]}\right)^\inv\subset \mathfrak{M}_{\mathbf{Q}'_0}((n/2,n/2),(0,1)),\quad \mbox{ if } n \mbox{ is even},$$
$$\left(\left(\CC^2\right)^{[n]}\right)^\inv\subset \bigcup_{\epsilon=\pm}\mathfrak{M}_{\mathbf{Q}'_0}(((n+\epsilon 1)/2,(n-\epsilon 1)/2),(0,1)),\quad \mbox{ if } n \mbox{ is odd}.$$

In  the case of even $n$ we observe that since the $\inv$-invariant  locus is not empty, the inclusion is actually an equality.
The fact that the corresponding quiver variety is the Hilbert scheme of points on the surface $\hat{A}_1$ is standard
(see, for example, Theorem 4.9 in \cite{kuz} ). 

In the case of odd $n$ we observe that the involution fixed locus must have at least two connected components. Indeed, if
$I\subset \CC[x,y]$ is an involution fixed ideal, then the quotient space $\CC[x,y]/I$ has an action of the involution and
the dimension $d(I):=\dim \left(\CC[x,y]/I\right)^\inv$ is constant along any connected component of $\left(\left(\CC^2\right)^{[n]}\right)^\inv$.
It is not hard to find two monomial ideals $I^\pm$ of codimension $n$ and $d(I^\pm)=(n\pm 1)/2$, these two ideals belong to two disjoint
connected components. Thus, we conclude that in the case of odd $n$, the inclusion is also an equality.

The formula for the dimension of the quiver varieties is standard and could be found, for example, in \cite{CrawleyBoevey01}.
\end{proof}

For small value of $n$, the result above has a more intuitive explanation. Indeed, if $n=2$, then the fixed locus
$\left((\CC^2)^{[2]}\right)^{\inv}$ is the closure of the locus consisting of the pairs of points $z,\inv(z)$, $z\ne (0,0)$.
On the other hand, if $n=3$, there are two connected components of the involution: the closure of the locus of triples
$(z,(0,0),\inv(z))$, $z\ne (0,0)$ and an isolated point which is the square of the maximal ideal $(x,y)^2$.
Then the connected components could be revealed by the analysis of the punctual Hilbert scheme $(\CC^2)^{[3]}_{(0,0)}$ which is
the cone over the twisted cubic. In more details \cite{Lehn}, the punctual Hilbert scheme parameterizes  ideals
\[I_{a,b,c,d,e}=(ax+by+ex^2,bx+cy+exy,cx+dy+ey^2,x^3,x^2y,xy^2,y^3),\]
where \((a:b:c:d:e)\in \mathbb{P}^4\) and \(ac=b^2\), \(ad=bc\), \(bd=c^2\) are the defining equations of the twisted cubic.

The involution \(\inv\) inverts variables \(a,b,c,d\) and thus fixed locus is the union of the vertex and the base of the cone.
The vertex of the cone corresponds to the ideal \((x^2,xy,y^2)\) that has no \(\inv\)-invariant deformations. The base \(e=0\) of the cone is parameterized
by the ideals generated by \(ax+by\) and \((x,y)^3\), the triple points on a line. One can deform such ideal to \(\inv\)-invariant ideal  \((ax+by,(x^2-f)x)\) if \(a\ne 0 \)
and \((ax+by,(y^2-f)y)\) if \(b\ne 0\).

Let us fix notations for the two connected components of the involution. Let \(U\subset \CC^2\) to a \(\inv\)-equivariant analytic neighborhood of \((0,0)\)  we denote the component of \(\left(\left(U\right)^{[n]}\right)^{\inv}\) of smaller dimension and large
dimension by:
\begin{equation*} 
  \left(\left(U\right)^{[n]}\right)^{\inv}_-,\quad \left(\left(U\right)^{[n]}\right)^{\inv}_+.
  \end{equation*}
These manifolds are used in the proof of theorem~\ref{counting}.



\section{Erratum}

In the \cite{KamenovaMongardiOblomkov19}, the main body of the current paper, the enumerative part of Theorems 1.1 and 1.3
is incorrect.  The source of error is an erroneous interpretation of the results of
\cite{CrawleyBoevey01} in the proof of Lemma B.1.
The error does not affect the enumeration of the top dimensional components
as in Theorem 3.7  and 3.9.
Moreover, the result of Theorem 1.3 does not apply to all involutions with a non trivial action on $H^3$, and we detail here the correct type of involutions considered.

The correct interpretation of the results of \cite{CrawleyBoevey01} yields statement
\[ ((\CC^2)^{[n]})^{\imath}=\bigcup_{n_1+n_2=n}\mathfrak{M}_{Q'_0}((n_1,n_2),(0,1)),\]
where \(n_1\ge 0,n_2>0\) and \(\mathfrak{M}_{Q'_0}((n_1,n_2),(0,1))\ne \emptyset\) iff
\(d(n_1,n_2)=n_2-(n_1-n_2)^2\ge 0\).  Moreover, the reflection functor isomorphisms developed
in \cite{CrawleyBoevey01} imply:
\[\mathfrak{M}_{Q'_0}((n_1,n_2),(0,1))\simeq (\widehat{A}_1)^{d(n_1,n_2)}.\]
In \cite{KamenovaMongardiOblomkov19}

Thus after correction we obtain the correct version of Lemma B.1

\begin{lem} For any \(n\) we have
  \[ ((\CC^2)^{[n]})^{\imath}=\bigsqcup_k (\widehat{A_1})^{[k]},\]
  where \(k=n_2-(n_1-n_2)^2\) with \(n_1+n_2=n\) and \(n_1\ge 0\), \(n_2>0\).
\end{lem}

The lower-dimensional component counts  are special cases of Theorems 3.0.3 and 3.0.7  from \cite{KamenovaMongardiOblomkov23}. To state the results let introduce the standard theta function:
\[\vartheta(\eta;q)=\sum_{n\in \ZZ}q^{n^2}\eta^n.\]

\begin{thm}
Let $S$ be a K3 surface and let $\imath$ be a symplectic involution of  $S$. Let $S^{[n]}$ be the Hilbert scheme of $n$ points 
on $S$ and let us consider the induced action of $\imath$ on $S^{[n]}$. 
Then all the irreducible components of the locus stabilized by $\imath$ 
are deformation equivalent to Hilbert schemes of points on a K3 surface or isolated points, 
and their number for each dimension $2k$ is 
$$ N_k= \Theta(n-2k),$$
where $\Theta(j)$ is the $j$-th coefficient of the Theta series 
\[\Theta(q)=\sum_{i\ge 0}\Theta(i)q^i=\vartheta(q;q^2)^8.\]
\end{thm}

\begin{oss}\label{rem:invol_types}
All involutions on generalized Kummer manifolds act trivially on the second cohomology, and they can be subdivided in three different geometrical categories:
\begin{itemize}
\item Involutions obtained by a translation of a two torsion point.
\item Involutions obtained by a sign change composed with a translation by a two torsion point.
\item Involutions obtained by a composition of a translation by a point of  order at least three with a sign change.
\end{itemize}
The first kind of involutions only appears when $n$ is odd.
\end{oss}
In this paper we deal with the second kind of involutions, and we call them \emph{regular} involutions. The fixed locus of the first kind of involutions was analyzed in \cite[Lemma 3.5]{ogu}, while the third kind was analyzed only in the case of Kummer sixfolds in \cite[Lemma 2.4]{floc}. It would be interesting to compute the fixed locus in general for an involution of the third kind.

\begin{thm} \label{thm:kum}
Let $X=K_{n}(A)$ be a  \(n\)-Kummer \hk manifold  and let $\imath\in Aut(A)$ 
be a regular symplectic involution of the abelian surface $A$.
Then all irreducible components of $X^\imath$ are of $K3^{[k]}$ type, and their 
number is 
$$ N_k= \Theta(n-2k;1),$$
where $\Theta(j;1)$ is the $j$-th coefficient of the Theta series 
\[\Theta(q)=\sum_{i\ge 0,\gamma\in A,\gamma^2=1}\Theta(i;\gamma)q^i\gamma=\prod_{\gamma\in A,\gamma^2=1}\vartheta(\gamma\cdot q;q^2).\]
\end{thm}

\begin{oss}
All three kinds of involutions  of Remark \ref{rem:invol_types} can be recognized by their action on higher cohomology: involutions obtained by pure translations have a trivial action on $H^3$, and involutions of the third kind, those obtained by composing a translation of order at least three with $-1$, are recognizable by their action on higher cohomology. As an example, if $n\equiv 3$ modulo $4$, in $H^6$ there are $256$ distinguished classes of codimension three subvarieties of a generalized Kummer. These subvarieties are in correspondance with points in $A[4]$, each of them is the locus of subschemes having support of multiplicity at least four on a given four torsion point of $A$. Regular involutions fix 16 of these subvarieties and permute the rest, while involutions of the third kind (obtained with an order four translation) freely permute these subvarieties.

\end{oss}

\section*{Acknowledgements for the Erratum}
We are grateful to Mirko Mauri for pointing out the inaccuracy in Theorem~\ref{thm:kum}. L.K. is partially supported by a grant from the Simons Foundation/SFARI (522730, LK). 
G.M. is supported by PRIN2020 research
grant 2020KKWT53 and by PRIN2022 research grant 2022PEKYBJ, and is a member of INdAM GNSAGA. A.O. is partially supported by NSF grant DMS-2200798 and 
NSF FGR grant DMS-1760373.


\appendix

\end{document}